\pdfoutput=1
\RequirePackage{fix-cm}%
\documentclass[paper=a4,fontsize=12pt,twoside,cleardoublepage=empty,listof=totoc,bibliography=totoc,index=totoc,BCOR13.0mm]{scrartcl}%
\usepackage[latin1]{inputenc}
\usepackage[T1]{fontenc}
\usepackage{lmodern,fixltx2e}%
\usepackage{etex}%
\usepackage[spanish,german,british]{babel}
\usepackage{hyphxcpt}
\newif\ifforceblacklinks\forceblacklinksfalse

\def\PublicationTitle{Unique inclusions of maximal
                      \texorpdfstring{\cclones{}}{C-clones}
                      in maximal clones}%
\def\CorrespondingAuthor{Mike Behrisch}%
\def\SecondAuthor{Edith Vargas-García}%

\def\TUWname{\foreignlanguage{german}{%
             Tech\-ni\-sche Uni\-ver\-si\-t\"{a}t Wien}}%
\def\InstitutCL{\foreignlanguage{german}{%
             In\-sti\-tut f\"{u}r Com\-pu\-ter\-spra\-chen}}%
\def\PostleitzahlWien{\mbox{A-1040} Vienna, Austria}%
\def\ULeeds{University of Leeds}%
\def\SchoolName{School of Mathematics}%
\def\EdithAddress{Woodhouse Lane,
                  Leeds
                  LS2 9JT,
                  UK}%
\def\grantno{207510}%
\def\PublicationTopic{relationship of maximal
                      \texorpdfstring{\cclones{}}{C-clones}
                      and maximal clones on finite domains}%
\def\PublicationKeywords{Clone,
                         \texorpdfstring{\cclone{}}{C-clone},
                         clausal relation,
                         maximal \texorpdfstring{\cclone{}}{C-clone},
                         maximal clone}%
\usepackage{amsmath}
\usepackage{amssymb}
\usepackage{amsfonts}
\usepackage{enumerate}
\usepackage{graphicx}                
\usepackage{url}                     

\usepackage{mathoperators}
\usepackage[uglyemptyset,hyperref]{usefulthings}
\usepackage{PolInv}
\usepackage[plain,proof,nocounter]{thmdefinitions}

\ChangePolInvDefaultCarrierSet{D}

\newcommand{\bfa}[1]{\ensuremath{{\mathbf{#1}}}}
\newcommand{\cclone}{\nbdd{\mathit{C}}clone}
\newcommand{\cclones}{\cclone{s}}
\newcommand{\ignoreAndReturnC}[1]{\ensuremath{\mathcal{C}}}
\DefineOpCloneClosureLikeCmd{\genCClone}{\ignoreAndReturnC} 
\DefineRelCloneClosureLikeCmd{\genRelCClone}{\ignoreAndReturnC} 

\DeclareMathOperator{\CInvOp}{\mathit{C}\InvOp}
\DefinePolLikeCmd{\CInv}{\CInvOp}
\DeclareMathOperator{\CRelOp}{R}
\newcommand{\Relab}[2]{\ensuremath{\CRelOp_{#2}^{#1}}}%
\newcommand{\Rab}{\Relab{\mathbf{a}}{\mathbf{b}}}%
\newcommand{\Ruab}{\Relab{(a)}{(b)}}%
\DeclareMathOperator{\CCloneSetOp}{\mathit{C}\CRelOp}
\newcommand{\CRd}[1]{\ensuremath{\CCloneSetOp_{#1}}}
\newcommand{\CR}{\CRd{D}}

\newcommand{\Dommenosa}[1]{\ensuremath{D\setminus\set{#1}}}

\newcommand{\interval}[2]{\ensuremath{\fapply{#1,#2}}}%
\newcommand{\name}[1]{\textsc{#1}}
\hypersetup{
    pdftitle={\PublicationTitle},    
    pdfauthor={\CorrespondingAuthor, \SecondAuthor},     
    pdfsubject={Preprint: \PublicationTopic},   
    pdfkeywords={\PublicationKeywords},         
}
\ifforceblacklinks%
\hypersetup{%
    linkcolor=black,     
    citecolor=black,     
    filecolor=black,     
    urlcolor=black       
}%
\fi%
\numberwithin{equation}{section} 

\begin{document}

\thispagestyle{empty}
\selectlanguage{british}

\title{\PublicationTitle}
\author{\CorrespondingAuthor%
        \thanks{\TUWname, \InstitutCL}%
        \and%
        \SecondAuthor%
        \thanks{\ULeeds, \SchoolName}
        \thanks{The research of the second named
                author was supported by CONACYT grant no.~\grantno.}}
\date{\today}
\maketitle

\begin{abstract}
\cclones{} are polymorphism sets of so\dash{}called clausal relations, a
special type of relations on a finite domain, which first appeared in
connection with constraint satisfaction problems in~\cite{%
CreignouHermannKrokhinSalzerComplexityOfClausalConstraintsOverChains}.
We completely describe the relationship \wrt\ set inclusion between
maximal \cclones{} and maximal clones. As a main result we obtain that for
every maximal \cclone{} there exists exactly one maximal clone in which it
is contained. A precise description of this unique maximal clone, as well
as a corresponding completeness criterion for \cclones{} is given.
\bgroup
\let\thefootnote\relax%
\footnote{%
\par%
}%
\setcounter{footnote}{0}%
\egroup
\end{abstract}

\section{Introduction}
\emph{Clones} are sets of operations on a fixed domain that are closed
under composition and contain all projections. The clones on a finite set
\m{D} are precisely the \name{Galois} closed sets of
operations~\cite{BodnarcukKaluzninKotovRomovGaloisTheoryForPostAlgebras}
with respect to the well\dash{}known \emph{\name{Galois}
connection} \m{\Pol{}{-}\Inv{}}
induced by the relation ``an operation \m{f} \emph{preserves} a relation
\m{\varrho}'' (see
also~\cite{PoeConcreteRepresentationOfAlgebraicStructuresAndGeneralGaloisTheory,PoeGeneralGaloisTheoryForOperationsAndRelations}).
In other words, every clone \m{F} on \m{D} can be described by
\m{F=\Pol{Q}} for some set \m{Q} of relations
(cf.\ Section~\ref{sect:preliminaries} for the notation).
\par

In this paper we continue the investigations from~\cite{BehVar}
and~\cite{Var} concerning clones on a finite set \m{D} described by
relations from a special set \m{\CR}. They are named
\emph{clausal relations} and were originally introduced in~\cite{%
CreignouHermannKrokhinSalzerComplexityOfClausalConstraintsOverChains}.
A clausal relation is the set of all tuples over \m{D} satisfying
disjunctions of inequalities of the
form \m{x\geq d} and \m{x\leq d}, where \m{x,d} belong to the finite set
\m{D=\set{0,1,\dotsc,n-1}}.
\par

We are interested in understanding the structure of clones that are
determined by sets of clausal relations, so\dash{}called
\emph{\cclones{}}. Their lattice has been delineated completely in
Theorem~2.14 of~\cite{Var} for the case that \m{\abs{D}=2}.
In this paper we study the co\dash{}atoms in the
lattice of all \cclones{}, the \emph{maximal \cclones{}}, for an
arbitrary finite set \m{D}. Since every clone on \m{D} either equals
\m{\Op{D}} (the set of all finitary operations on \m{D}) or is contained
in some \emph{maximal clone} (co\dash{}atom of the lattice of all clones)
(see \eg~\cite[Haupt\-satz~3.1.5, p.~80; Voll\-stän\-dig\-keits\-kri\-te\-ri\-um~5.1.6, p.~123]{PoeKal}
or \cite[Proposition~1.15, p.~27]{SzendreiClonesInUniversalAlgebra}), our
aim is to investigate  which maximal \cclones{} are contained in which
maximal clones. We achieve a complete description in
Theorem~\ref{thm:char-relationship-max-clones-max-cclones} and
thereby answer the question that was left open
in~\cite{BehVarMaxCclonesPreprint}.
\par

Using \name{Rosenberg}'s theorem
(see Theorem~\ref{thm:Rosenberg-classification} below),
all maximal clones on \m{D} can be classified into six types.
From~\cite{BehVarMaxCclonesPreprint} we know already that a few of them,
\eg\ centralisers of prime permutations, polymorphism sets of an affine,
of a central relation of arity at least three or of an \nbdd{h}regular
relation, do not contain any maximal \cclone. We shall see that this
phenomenon extends to maximal clones of monotone functions with regard
to some bounded partial order whenever \m{\abs{D}\geq 3}.
\par

To our surprise, it turns out that every maximal \cclone\ is contained in
a unique maximal clone, either given as polymorphism set of a
non\dash{}trivial equivalence relation or a unary or binary central
relation (vide infra for a definition of such relations).
The respective details can be seen from our main result,
Theorem~\ref{thm:char-relationship-max-clones-max-cclones}. As a corollary
we also deduce a new completeness criterion for \cclones.
\par

We start by introducing our notation, recalling some fundamental facts
about the \name{Galois} theory for clones, the characterisation of maximal
clones and \cclones, respectively, and providing two basic lemmas in
Section~\ref{sect:preliminaries}. Then we devote one section each to
examine possible inclusions of maximal \cclones{} in maximal clones of
the form \m{\Pol{\rho}}, where \m{\rho} is a non\dash{}trivial unary
relation, a bounded partial order relation, a non\dash{}trivial
equivalence relation or an at least binary central relation. Finally, in
Section~\ref{sect:thm-statement}, we deduce our main theorem from the
previous results.

\section{Main notions and preliminaries}\label{sect:preliminaries}

Throughout the text, \m{D} will denote the finite non\dash{}empty set
\m{\set{0,\dotsc,n-1}} (\m{n>0}) and \m{\N=\set{0,1,2,\dotsc}} the set of
\emph{natural numbers}. We put \m{\Np\defeq\N\setminus\set{0}}. Moreover,
for a function \m{\functionhead{f}{A}{B}}, we denote its \emph{image} by
\m{\im\apply{f}\defeq \lset{f\apply{x}}{x\in A}}.
\par

Let \m{m \in \Np}. An \emph{\nbdd{m}ary relation} \m{\varrho} on \m{D}
is a subset of the \nbdd{m}fold Cartesian product \m{D^{m}}.
By \m{\Rel[m]{D}\defeq\powerset{D^m}} we denote the
\emph{set of all \nbdd{m}ary relations} on \m{D} and by
\m{\Rel{D}\defeq\bigcup_{m\in\Np}\Rel[m]{D}}
the \emph{set of all finitary relations} on \m{D}.
For a binary relation \m{\rho\subs D^2} we denote its \emph{inverse} by
\m{\rho^{-1}\defeq \lset{\apply{y,x}}{\apply{x,y}\in\rho}}.
\par

We want to study clones that are determined by sets of clausal relations.
Even though, for almost all results, we will need only binary clausal
relations, we define them here in full generality.
\begin{definition}\label{def:clausal-relation}
Let \m{p,q \in \Np}.
For given parameters \m{\bfa{a}=\apply{\liste{a}{p}}\in D^p}
and \m{\bfa{b}=\apply{\liste{b}{q}}\in D^q}, the \emph{clausal
relation} \m{\Rab} \emph{of arity} \m{p+q} is the
set of all tuples \m{\apply{\liste{x}{p},\liste{y}{q}} \in D^{p+q}}
satisfying
\begin{equation*}\label{ecuationclau}
(x_1\geq a_1)\vee \dotsm \vee (x_p\geq a_p)%
\vee(y_1\leq b_1)\vee \dotsm \vee(y_q\leq b_q ).
\end{equation*}
In this expression \m{\leq} denotes the canonical linear order on \m{D}
and \m{\geq} its dual.
\end{definition}

For \m{k \in \Np} we denote by
\m{\Op[k]{D}\defeq\rset{f}{f\colon D^{k}\to D}}
the \emph{set of all \nbdd{k}ary operations on \m{D}} and
by \m{\Op{D}\defeq\bigcup_{k\in\Np}\Op[k]{D}} the \emph{set of all
finitary operations on \m{D}}.\par

Next, we will consider  a \name{Galois} connection between sets of
operations and relations that is based on the so\dash{}called
\emph{preservation relation}. It is the most important tool for our
investigations.

\begin{definition}\label{def:preservation}
Let \m{m,k \in \Np}.
We say that a \nbdd{k}ary operation \m{f\in \Op[k]{D}}
\emph{preserves} an \nbdd{m}ary relation \m{\varrho \in \Rel[m]{D}},
denoted by \m{f\preserves \varrho}, if whenever
\[r_1=\apply{a_{11},\dotsc,a_{m1}}\in \varrho,
  \dotsc,
  r_k=\apply{a_{1k},\dotsc,a_{mk}}\in \varrho,\]
it follows that also \m{f} applied to these tuples belongs to
\m{\varrho}, \ie\
\[f \circ \apply{r_1,\dotsc,r_k}\defeq
\apply{f\apply{a_{11},\dotsc, a_{1k}},
       \dotsc,
       f\apply{a_{m1},\dotsc,a_{mk}}} \in \varrho.\]
\end{definition}
\par

For a set of operations \m{F\subseteq \Op{D}}, we
denote by \m{\Inv{F}} the set of all relations that are \emph{invariant}
for all operations \m{f\in F}, \ie\
\m{\Inv{F} \defeq\lset{\varrho \in \Rel{D}}{\forall f\in F\colon f
\preserves \varrho}}.
Similarly, for a set \m{Q \subseteq \Rel{D}} of relations,
\m{\Pol{Q}\defeq
         \lset{f\in F}{\forall \varrho \in Q\colon f \preserves \varrho}}
denotes the \emph{set of polymorphisms of \m{Q}}.
Furthermore, for \m{k\in \Np} we
abbreviate \m{\Pol[k]{Q}\defeq \Op[k]{D}\cap \Pol{Q}}.
Usually, we will write
\m{\Pol{\rho}} for \m{\Pol{\set{\rho}}}, \m{\rho \in \Rel{D}} and
\m{\Inv{f}} for \m{\Inv{\set{f}}}, \m{f \in \Op{D}}. The operators
\m{\Pol{}} and \m{\Inv{}} define the \name{Galois} connection
\m{\Pol{}{-}\Inv{}}.
\par

On a finite set \m{D} the \name{Galois} closed sets of
relations~\cite{GeigerClosedSystemsOfFunctionsAndPredicates} with
respect to \m{\Pol{}{-}\Inv{}} are exactly the so\dash{}called
\emph{relational clones}. These can be characterised as those sets of
finitary relations on \m{D} that are closed \wrt\ \emph{primitive
positively definable relations}, \ie\ those arising as interpretations of
first order formulæ where only predicate symbols corresponding to
relations from \m{Q}, falsity, variable identifications, finite
conjunctions and finite existential quantification are allowed. For a set
\m{Q\subs\Rel{D}} of relations, we denote by \m{\genRelClone{Q}} the
closure of \m{Q} with regard to such formulæ, which equals the least
relational clone generated by \m{Q}, \ie, by the above, we have
\m{\genRelClone{Q}=\Inv{\Pol{Q}}}.
\par

A relation \m{\rho\in\Rel{D}} is called \emph{trivial} if it is preserved
by every function, \ie\ if \m{\Pol{\rho}=\Op{D}}, or equivalently
\m{\rho\in\Inv{\Op{D}}}. The set of trivial relations \m{\Inv{\Op{D}}} can
be characterised to contain precisely all so\dash{}called \emph{diagonal
relations} (see
\eg~\cite[3.2~Definitions (R0), p.~25]{%
PoeGeneralGaloisTheoryForOperationsAndRelations}
or~\cite[p.~5]{BehClonesWithNullaryOperations}
for a definition), which are generalisations of
the binary diagonal relations \m{\Delta=\lset{\apply{x,x}}{x\in D}} and
\m{\nabla=D\times D}.
\par

A set  \m{F\subseteq \Op{D}} of operations is called a \emph{\cclone}
if \m{F=\Pol{Q}} for some set \m{Q} of clausal relations. All \cclones{}
on \m{D}, ordered by set inclusion, form a complete lattice, whose
co\dash{}atoms are called \emph{maximal \cclones}.
\par

From~\cite{Edith-thesis} we have a description of all maximal \cclones{}
on finite sets as polymorphism sets of binary clausal relations
\m{\Ruab=\lset{\apply{x,y}\in D^2}{x\geq a \,\vee\, y\leq b}}.
\begin{theorem}[\cite{Edith-thesis}]\label{teo_caracterization_cclonemax}
Let \m{M\subseteq \Op{D}} be a \cclone{}. \m{M} is maximal if and only if
there are elements \m{a\in \Dommenosa{0}} and \m{b \in \Dommenosa{n-1}}
such that \m{M= \Pol{\Ruab}}.
\end{theorem}\par

Likewise, the following characterisation of maximal clones on finite sets
is well known. The sorts of relations occurring in
Theorem~\ref{thm:Rosenberg-classification} will be defined below as far as
they are needed for later purposes.

\begin{theorem}[\cite{RosenbergMaximalClonesFrench,%
                      RosenbergMaximalClones}]%
\label{thm:Rosenberg-classification}
A clone \m{F\subs\Op{D}} is maximal if and only if it is of the
form \m{\Pol{\rho}}, where \m{\rho} is a non\dash{}trivial relation
belonging to one of the following classes:
\begin{enumerate}
 \item The set of all partial orders with least and greatest element.
 \item The set of all graphs of prime permutations.
 \item The set of all non\dash{}trivial%
       \footnote{%
       Here \emph{non\dash{}trivial} means
       \m{\Eq(D)\setminus\set{\Delta,\nabla}}.}
       equivalence relations.
 \item\label{case:affine-relations}
       The set of all affine relations \wrt\ some elementary
       \name{Abel}ian \nbdd{p}group on \m{D} for some prime \m{p}.
 \item The set of all central relations of arity \m{h}
       (\m{1 \leq h < \abs{D}}).
 \item The set of all \nbdd{h}regular relations
       (\m{3 \leq h \leq \abs{D}}).
\end{enumerate}
\end{theorem}\par

In~\cite{BehVarMaxCclonesPreprint} it has been shown that
\m{\Pol{\Ruab}\not\subs\Pol{\rho}} whenever \m{\rho} is the
graph of a prime permutation, an affine relation \wrt\ some elementary
\name{Abel}ian \nbdd{p}group or an at least ternary central or
\nbdd{h}regular relation. The remaining types of relations from
\name{Rosenberg}'s theorem are bounded orders, non\dash{}trivial
equivalences and unary and binary central relations.
\par

A \emph{central relation} is a totally symmetric, totally reflexive
relation having a central element and not being a diagonal relation.
\emph{Total symmetry} means closure under all permutations of entries of
tuples; \emph{total reflexivity} requires that every tuple having two
identical entries has to belong to the relation. An element \m{c\in D} is
\emph{central} for \m{\rho} if any tuple containing \m{c} as an entry is a
member of \m{\rho}.
\par

The only unary diagonal relations are \m{\emptyset} and \m{D}, the binary
ones are \m{\Delta} and \m{D\times D}.
Therefore, unary central relations are precisely all subsets
\m{\emptyset\subsetneq\rho\subsetneq D}. Binary central relations can be
described as follows. Note that for binary relations the notions of
total symmetry and total reflexivity coincide with ordinary symmetry and
reflexivity, respectively. For \m{c\in D} let
\m{\rho_{c}\defeq \Delta \cup \apply{\set{c}\times D}
                         \cup \apply{D\times\set{c}}} and
\m{A_c\defeq \lset{\apply{x,y}\in D^2\setminus\rho_{c}}{x<y}
 = \lset{\apply{x,y}\in\apply{\Dommenosa{c}}^2}{x<y}}.
For any \m{S_c\ovflhbx{1.24381pt}\subsetneq\ovflhbx{1.24380pt} A_c}
we have a binary central relation
\m{\rho_{c,S_c}\defeq \rho_{c}\cup S_c \cup S_c^{-1}}, and it is easy to
see that all of them arise in this way. Note that for \m{n=\abs{D}=3} we
always have \m{S_c=\emptyset} as \m{A_c} contains only one pair.
\par

Supposing \m{\abs{D}\geq 3}, the goal of the following sections is to
understand completely, for which parameters  \m{a\in\Dommenosa{0}},
\m{b\in\Dommenosa{n-1}} and which relations \m{\rho} from
Theorem~\ref{thm:Rosenberg-classification} we have the inclusion
\m{\Pol{\Ruab}\subs\Pol{\rho}}.
\par

To realise this, we may want to use unary functions
\m{f\in\Pol[1]{\Ruab}\setminus\Pol{\rho}} as witnesses for
\m{\Pol{\Ruab}\not\subs\Pol{\rho}}, where \m{\Pol{\rho}} is a maximal
clone. The following lemma gives a simple sufficient condition for
functions \m{f\in\Op[1]{A}} to preserve \m{\Ruab}.
\par

\begin{lemma}\label{lem:pres-of-max-clausal-rel}
For \m{a,b\in D} and every \m{f\in\Op[1]{D}} such that
\m{\im\apply{f}\subs\set{0,\dotsc,b}} or dually
\m{\im\apply{f}\subs\set{a,\dotsc,n-1}}, we always have
\m{f\in\Pol{\Ruab}}.
\end{lemma}
\begin{proof}
If \m{\im\apply{f}\subs\set{0,\dotsc,b}}, then we have \m{f(y)\leq b} for
all \m{\apply{x,y}\in\Ruab} and so
\m{f\preserves\Ruab}. If \m{\im\apply{f}\subs\set{a,\dotsc,n-1}}, then
likewise \m{f\apply{x}\geq a} for all \m{\apply{x,y}\in\Ruab} and also
\m{f\preserves\Ruab}.
\end{proof}
\par

When constructing unary functions
\m{f\in\Pol[1]{\Ruab}\setminus\Pol{\rho}} as witnesses for
non\dash{}inclusions \m{\Pol{\Ruab}\not\subs\Pol{\rho}}, where
\m{\Pol{\rho}} is a maximal clone, it is helpful to know how much choice
we have for \m{f}. We cannot achieve a converse to
Lemma~\ref{lem:pres-of-max-clausal-rel}, but the following result seems to
be as good as we can get in this respect.
\par
\begin{lemma}\label{lem:unary-witnesses}
For \m{a,b\in D} and every \m{f\in\Pol[1]{\Ruab}} the
following conditions hold:
\begin{enumerate}[(a)]
\item\label{item:f-pres-b-or-im-a}
      \m{f\preserves \set{0,\dotsc,b}} or
      \m{\im\apply{f}\subs\set{a,\dotsc,n-1}}.
\item\label{item:f-pres-a-or-im-b}
      \m{f\preserves \set{a,\dotsc,n-1}} or
      \m{\im\apply{f}\subs\set{0,\dotsc,b}}.
\item\label{item:f-pres-a-or-b}
      \m{f\preserves \set{a,\dotsc,n-1}} or
      \m{f\preserves \set{0,\dotsc,b}}.
\end{enumerate}
\end{lemma}
\begin{proof}
Statement~\eqref{item:f-pres-a-or-b} follows
from~\eqref{item:f-pres-b-or-im-a} since the condition
\m{\im\apply{f}\subs\set{a,\dotsc,n-1}} implies
\m{f\preserves\set{a,\dotsc,n-1}}. The proof of
statement~\eqref{item:f-pres-a-or-im-b} is dual to that
of~\eqref{item:f-pres-b-or-im-a}, so we only deal with the latter one. If
\m{f\mathrel{\npres}\set{0,\dotsc,b}}, then there exists some \m{y\leq b}
such that \m{f\apply{y}>b}. This means we have \m{\apply{x,y}\in\Ruab}
for all \m{x\in D}. Since \m{f\preserves \Ruab}, we obtain
\m{\apply{f(x),f(y)}\in\Ruab}, \ie\ \m{f(x)\geq a} due to
\m{f(y)>b}.
\end{proof}
\par

Using more sophisticated constructions of binary witnesses, we will first
be attacking the case of maximal clones \m{\Pol{\rho}} given by
non\dash{}trivial unary relations \m{\emptyset\subsetneq\rho\subsetneq D}.

\section{\texorpdfstring{Non\dash{}trivial unary relations}{%
         Non-trivial unary relations}}%
\label{sect:unary-rels}

The following lemma gives sufficient conditions for binary operations to
belong to a given maximal \cclone{}.
\begin{lemma}\label{lem:binary-f-pres-Rab}
Let \m{a,b\in D} and suppose \m{f\in\Op[2]{D}} satisfies
\m{f\apply{x,y}\leq b} for all pairs
\bgroup\def\sqz{\ovflhbx{0.5pt}}%
\m{\apply{x,y}\sqz\in\sqz D^2} where \m{x\sqz\leq\sqz b} or
\m{y\sqz\leq\sqz b}, and \m{f\sqz\sqz\apply{x,y}\sqz\geq\sqz a} for all
\m{\apply{x,y}\sqz\in\sqz D^2}
where \m{x,y\sqz\geq\sqz a}.\egroup
\footnote{Such functions exist most easily, if
\m{a>b}, but also for \m{a\leq b}.} Then \m{f\in\Pol{\Ruab}}.
\par
Dually, if $f(x,y)\geq a$ for all $(x,y)\in D^2$ such that
\m{x\geq a} or
\m{y\ovflhbx{1pt}\geq\ovflhbx{1pt} a}, and
\m{f\ovflhbx{1pt}\apply{x,y}\ovflhbx{1pt}\leq\ovflhbx{1pt} b}
for those pairs \m{\apply{x,y}\in D^2} where \m{x,y\leq b}, then
\m{f\in\Pol{\Ruab}}, too.
\end{lemma}
\begin{proof}
Let \m{\apply{x_1,y_1},\apply{x_2,y_2}\in\Ruab}. If
\m{f\apply{y_1,y_2}\leq b}, then
\m{\apply{f\apply{x_1,x_2},f\apply{y_1,y_2}}} belongs to
\m{\Ruab} and we are
done. Else, by the assumption on \m{f} we must have \m{y_1,y_2>b}, which
implies \m{x_1,x_2\geq a} due to
\m{\apply{x_1,y_1},\apply{x_2,y_2}\in\Ruab}. Therefore,
\m{f\apply{x_1,x_2}\geq a}, which implies again
\m{\apply{f\apply{x_1,x_2},f\apply{y_1,y_2}}\in\Ruab}.
This proves that \m{f\in\Pol{\Ruab}}.
The proof of the second claim is by dualisation.
\end{proof}

We can use this type of functions to witness non\dash{}inclusions of
maximal \cclones\xspace in maximal clones given by a non\dash{}trivial
unary relation \m{\rho} whenever there exists some \m{x\in\rho}
respecting \m{b<x<a}.

\begin{corollary}\label{cor:unary-nonincl-for-b<x<a}
Let \m{a,b\in D} and suppose \m{\rho\subsetneq D} contains an element
\m{x\in\rho} such that \m{b<x<a}. Every binary function \m{f\in\Op[2]{D}}
satisfying one of the conditions from Lemma~\ref{lem:binary-f-pres-Rab}
and mapping \m{f\apply{x,x}=y} where \m{y\in D\setminus\rho} fulfils
\m{f\in\Pol[2]{\Ruab}\setminus\Pol{\rho}}. Such functions exist indeed,
whence we have \m{\Pol{\Ruab}\not\subseteq \Pol{\rho}}.
\end{corollary}
\begin{proof}
Since \m{f\in\Op[2]{D}} fulfils the conditions of
Lemma~\ref{lem:binary-f-pres-Rab}, we get \m{f\in\Pol{\Ruab}}; further,
the assumption \m{f\apply{x,x}=y} where \m{x\in \rho} and \m{y\notin\rho}
ensures that \m{f\notin\Pol{\rho}}.
\par
For the existence of such operations, verify that the following function
is well\dash{}defined due to \m{b<x<a}: we put \m{f\apply{u,v}\defeq a} if
\m{u,v\geq a}, \m{f\apply{x,x}\defeq y\notin\rho} and
\m{f\apply{u,v}\defeq 0\leq b} everywhere else. So \m{f} satisfies the first
condition from Lemma~\ref{lem:binary-f-pres-Rab}.
\end{proof}

In the next step we derive a necessary condition concerning the form of
the unary relation \m{\rho} that has to hold if
\m{\Pol{\Ruab}\subs\Pol{\rho}}.
\begin{lemma}\label{lem:unary-rels-cond}
For \m{a,b\in D} and a non\dash{}empty unary relation
\m{\emptyset\subsetneq\rho\subs D}, the inclusion
\m{\Pol{\Ruab}\subs\Pol{\rho}} implies
\m{\set{0,\dotsc,b}\cup \set{a,\dotsc,n-1}\subs \rho}.
\end{lemma}
\begin{proof}
If there existed some \m{x\leq b} such that \m{x\notin \rho}, then
\m{c_{x}\in\Pol{\Ruab\setminus\Pol{\rho}}}, where \m{c_x} denotes the
unary constant with value \m{x}, would contradict
the assumption \m{\Pol{\Ruab}\subs\Pol{\rho}}.
For \m{x\geq a} not belonging to \m{\rho} we use a similar argument.
\end{proof}

As a partial converse the next result establishes a sufficient condition
for an inclusion of a maximal \cclone\xspace in a maximal clone given by a
non\dash{}trivial unary relation.
\par

\begin{lemma}\label{lem:Rab-entails-0b-an-1}
Let \m{a,b\in D} such that \m{a>b}. Then we have
\begin{align*}
\Ruab\cap\apply{\Ruab}^{-1}
    &=\set{0,\dotsc,b}^2\cup\set{a,\dotsc,n-1}^2 \text{ and}\\
\lset{x\in D}{%
      \apply{x,x}\in\Ruab\cap\apply{\Ruab}^{-1}}
    &=\set{0,\dotsc,b}\cup\set{a,\dotsc,n-1},
\end{align*}
whence
\m{\Pol{\Ruab}\subs\Pol{\set{0,\dotsc,b}\cup\set{a,\dotsc,n-1}}}.
\end{lemma}
\begin{proof}
The second equality stated in the lemma will follow by variable
identification from
\m{\Ruab\cap\apply{\Ruab}^{-1}
    =\set{0,\dotsc,b}^2\cup\set{a,\dotsc,n-1}^2}.
In this equality the inclusion ``\m{\sups }'' is evident, so let us now consider
\m{\apply{x,y}\in\Ruab\ni\apply{y,x}}.
If \m{x\geq a > b}, then \m{\apply{y,x}\in\Ruab} implies
\m{y\geq a}, thus, \m{\apply{x,y}\in \set{a,\ldots,n-1}^2}.
Otherwise, we have \m{x<a}, such that \m{y\leq b<a} due to
\m{\apply{x,y}\in\Ruab}. So it follows \m{x\leq b} as \m{y<a} and
\m{\apply{y,x}\in\Ruab}. Hence, \m{\apply{x,y}\in\set{0,\dotsc,b}^2}.
\par

The second equality in the lemma implies
\m{\set{0,\dotsc,b}\cup\set{a,\dotsc,n-1}\in\genRelClone{\Ruab}},
and therefore,
\m{\Pol{\apply{\set{0,\dotsc,b}\cup\set{a,\dotsc,n-1}}}
\sups\Pol{\genRelClone{\Ruab}}
                 =\Pol{\Ruab}}.
\end{proof}

The following lemma solves the task for non\dash{}trivial unary relations.
\begin{lemma}\label{lem:incl-unary-rels}
Let \m{a,b\in D} and \m{\emptyset\subsetneq \rho\subsetneq D} be a unary
non\dash{}trivial relation. Then the inclusion
\m{\Pol{\Ruab}\subs\Pol{\rho}} holds if and only if
\m{\rho=\set{0,\dotsc,b}\cup\set{a,\dotsc,n-1}}
and \m{a-b\geq 2}.
\end{lemma}
\begin{proof}
If \m{\rho=\set{0,\dotsc,b}\cup\set{a,\dotsc,n-1}} and \m{a-b\geq 2>0},
then Lemma~\ref{lem:Rab-entails-0b-an-1} implies the inclusion
\m{\Pol{\Ruab}\subs\Pol{\rho}}. Conversely, if we assume this condition,
then Lemma~\ref{lem:unary-rels-cond} entails
\m{\set{0,\dotsc,b}\cup\set{a,\dotsc,n-1}\subs\rho}. If this inclusion were
proper, then there would exist some \m{x\in\rho} such that \m{x\not\leq b} and
\m{x\not\geq a}, \ie\ \m{b<x<a}. Since \m{\rho\subsetneq D},
Corollary~\ref{cor:unary-nonincl-for-b<x<a} yields a contradiction to the
assumption \m{\Pol{\Ruab}\subs\Pol{\rho}}. Therefore, we
have \m{\set{0,\dotsc,b}\cup\set{a,\dotsc,n-1}=\rho}.
Moreover, if \m{a-b\leq 1}, then we would have the full relation \m{\rho=D},
violating our assumption.
\end{proof}

\section{The case of bounded order relations}%
\label{sect:bounded-orders}

A \emph{bounded} (partial) order relation is an order relation having
both, a largest (top) element \m{\top}, and a least (bottom) element
\m{\bot}. If \m{\mathord{\preceq}\subs D^2} is an order relation on \m{D},
considered to be clear from the context, and \m{a,b\in D} are any two
elements, we occasionally use the notation
\m{\interval{a}{b}\defeq\lset{x\in D}{a\preceq x\preceq b}} and call it
the \emph{interval from \m{a} to \m{b}}. Clearly, if \m{a\not\preceq b},
then \m{\interval{a}{b}=\emptyset}.
\par

In the first step we construct binary functions witnessing
non\dash{}inclusions of certain maximal \cclones\xspace in maximal clones
described by non\dash{}trivial binary reflexive relations.
\begin{lemma}\label{lem:a-b-geq2-reflexive}
Assume that \m{a-b\geq 2}. Any operation \m{g\in\Op[2]{D}} satisfying
\m{g\apply{x,y}\leq b} when\-ever \m{y\leq b} and \m{g\apply{x,y}\geq a} for
all \m{\apply{x,y}\in D^2} where \m{y\geq a}, preserves \m{\Ruab}.
\par

Moreover, let \m{\rho\subsetneq D^2} be reflexive,
\m{\apply{x,y}\in\rho\setminus\Delta}, \m{\apply{u,v}\in D^2\setminus\rho},
\m{b<z<a}, and suppose, in addition to the above, that \m{g\apply{x,z}=u}
and \m{g\apply{y,z}=v}. Then we have
\m{g\in\Pol{\Ruab}\setminus\Pol{\rho}}.
\end{lemma}
\begin{proof}
First, we check that \m{g\in\Pol{\Ruab}}. Namely, if
\m{\apply{x_1,y_1},\apply{x_2,y_2}\in\Ruab} and \m{x_2\geq a}, then
\m{g\apply{x_1,x_2}\geq a}. Otherwise, we have \m{x_2<a} and \m{y_2\leq b},
which implies \m{g\apply{y_1,y_2}\leq b}. In both cases we obtain
\m{\apply{g\apply{x_1,x_2},g\apply{y_1,y_2}}\in\Ruab}.
\par
Furthermore, we have \m{\apply{x,y},\apply{z,z}\in\rho}, but
\m{\apply{g\apply{x,z},g\apply{y,z}}=\apply{u,v}\notin\rho}, proving
\m{g\mathrel{\npres}\rho}.
\end{proof}

If \m{a-b\geq 2}, the many requirements on the binary function in the
previous lemma are actually satisfiable.
\begin{corollary}\label{cor:a-b-geq2-reflexive}
For all \m{a,b\in D} such that \m{a-b\geq 2} and every
non\dash{}trivial binary reflexive relation
\m{\Delta\subsetneq \rho\subsetneq D^2}, we have
\m{\Pol[2]{\Ruab}\not\subs\Pol{\rho}}.
\end{corollary}
\begin{proof}
Since \m{a-b\geq 2}, binary functions \m{g} fulfilling the assumptions of
Lemma~\ref{lem:a-b-geq2-reflexive} are indeed constructible. Choosing
pairs \m{\apply{x,y}\in\rho\setminus\Delta} and
\m{\apply{u,v}\in D^2\setminus\rho}, we may, for instance, define
\m{g\apply{w,z}\defeq 0\leq b} for \m{z\leq b},
\m{g\apply{w,z}\defeq n-1\geq a} for \m{z\geq a},
\m{g\apply{w,z}\defeq u} for \m{b<z<a} and \m{w=x}, and
\m{g\apply{w,z}\defeq v} else, \ie\ for all \m{\apply{w,z}\in D^2}
satisfying \m{b<z<a} and \m{w\neq x}. Since \m{y\neq x}, this ensures that
\m{g\apply{y,z}=v} for all \m{b<z<a}, and hence \m{g} fulfils the
conditions of Lemma~\ref{lem:a-b-geq2-reflexive}.
\end{proof}

So the preceding result demonstrates that inclusions
\m{\Pol{\Ruab}\subs\Pol{\rho}} are impossible whenever \m{a-b\geq 2} and
\m{\rho} is a non\dash{}trivial equivalence or a bounded order relation.
In order to exclude more inclusions, we will use the following trivial
observation.
\begin{lemma}\label{lem:unary-f-pres-a-b-pres-Rab}
If for \m{a,b\in D} an operation
\m{f\in\Op[1]{D}} preserves the sets \m{\set{0,\dotsc,b}} and
\m{\set{a,\dotsc,n-1}}, then \m{f\preserves \Ruab}.
In particular this follows, if \m{a\leq b} and \m{f} preserves the sets
\m{\lset{x\in D}{x<a}}, \m{\lset{x\in D}{a\leq x\leq b}} and
\m{\lset{x\in D}{b<x}}.
\end{lemma}
\begin{proof}
If \m{\apply{x,y}\in\Ruab} and \m{x\geq a}, then \m{f\apply{x}\geq a},
otherwise, \m{x<a} and \m{y\leq b}, whence \m{f\apply{y}\leq b}. In both
cases we have \m{\apply{f\apply{x},f\apply{y}}\in\Ruab}. The additional
remark follows since for \m{a\leq b} the union of the first two mentioned
sets is \m{\set{0,\dotsc,b}}, the union of the last two sets is
\m{\set{a,\dotsc,n-1}}, and invariant relations of unary operations are
closed under arbitrary unions of relations of identical arity.
\end{proof}

We shall use transpositions that preserve the subsets
\m{\set{0,\dotsc,b}} and
\m{\set{a,\dotsc,n\ovflhbx{1.02675pt}-\ovflhbx{1pt}1}} from
Lemma~\ref{lem:unary-f-pres-a-b-pres-Rab} in
Proposition~\ref{prop:bounded-orders} below. However, first, we shall deal
with a few exceptional cases. They are actually variations of one case up
to different dualisations, but we consider them explicitly here.

\begin{lemma}\label{lem:orders-exceptions}
Let \m{n\geq 3}, \m{a,b\in D} and \m{\mathord{\preceq}\subs D^2} be a
bounded order relation with least element \m{\bot} and greatest element
\m{\top}. If
\begin{enumerate}[(a)]
\item\label{item:R11}
      \m{0=\bot<1 = a = b = \top}, or
\item\label{item:12}
      \m{0=\bot<1= a}, \m{n-2=b<n-1=\top}, or
\item\label{item:31}
      \m{n-1=\bot>n-2=b}, \m{1=a>0=\top}, or
\item\label{item:Rn-2n-2}
      \m{n-1=\bot>n-2=b=a=\top}, or
\item\label{item:21}
      \m{a=\bot=b=1>0=\top}, or
\item\label{item:22}
      \m{a=\bot=b=n-2<n-1=\top},
\end{enumerate}
then there exists some
\m{f\in\Pol[1]{\Ruab}\setminus\Pol{\mathord{\preceq}}}, whence
\m{\Pol{\Ruab}\subs\Pol{\mathord{\preceq}}} is impossible.
\end{lemma}
\begin{proof}
In each of the cases we explicitly define a unary operation
\m{f\in\Op[1]{D}} not preserving \m{\mathord{\preceq}}. The condition
\m{f\in\Pol{\Ruab}}
will always follow from Lemma~\ref{lem:pres-of-max-clausal-rel}.

\begin{enumerate}[(a)]
\item Define \m{f\in\Op[1]{D}} by \m{f(0)\defeq 1}, \m{f(x)\defeq x} for
      \m{x\in \Dommenosa{0}}. Since \m{n\geq 3} there exists some
      element \m{x\in D\setminus\set{0,1}}. We have \m{0=\bot \preceq x},
      but the assumption \m{\top = 1=f(0)\preceq f(x)=x} would imply
      the contradiction \m{x=\top=1}, so
      \m{f\mathrel{\npres}\mathord{\preceq}}.
      Besides, \m{\im\apply{f}=\Dommenosa{0}=\set{a,\dotsc,n-1}}, so
      \m{f\preserves \Ruab}.
\item Define \m{f\in\Op[1]{D}} by \m{f(n-1)\defeq 0} and \m{f(x)\defeq x}
      for \m{x\in \Dommenosa{n-1}}. Since
      \m{\im\apply{f}=\Dommenosa{n-1} =\set{0,\dotsc,b}}, we get
      \m{f\preserves\Ruab}. We have \m{1\preceq \top=n-1} and \m{1<n-1}
      due to \m{n\geq 3}, so supposing \m{1=f(1)\preceq f(n-1)=0=\bot}
      would imply the contradiction \m{1=\bot=0}. Hence,
      \m{f\mathrel{\npres}\mathord{\preceq}}.
\item Define \m{f\in\Op[1]{D}} by \m{f(0)\defeq n-1} and \m{f(x)\defeq x}
      for \m{x\in\Dommenosa{0}}. Evidently,
      \m{\im\apply{f}=\Dommenosa{0}=\set{a,\dotsc,n-1}}, so
      \m{f\preserves\Ruab}. We have \m{1\preceq \top=0}, and assuming
      \m{1=f(1)\preceq f(0)=n-1=\bot} would imply \m{1=\bot=n-1},
      \ie\ \m{n=2}.
      Thus, \m{f\mathrel{\npres}\mathord{\preceq}}.
\item Define \m{f\in\Op[1]{D}} by \m{f\apply{n-1}\defeq n-2} and
      \m{f(x)\defeq x} for \m{x\in\Dommenosa{n-1}}. For \m{n\geq 3}, there
      exists some \m{x\in \Dommenosa{n-1,n-2}}. We have
      \m{n-1=\bot\prec x}, but \m{\top=n-2=f(n-1)\preceq f(x)=x}
      would imply
      \m{x=\top=n-2}, whence \m{f\mathrel{\npres}\mathord{\preceq}}.
      Clearly, \m{\im\apply{f}=\Dommenosa{n-1}=\set{0,\dotsc,b}}, thus
      \m{f\in\Pol{\Ruab}}.
\item Define \m{f\in\Op[1]{D}} as in~\eqref{item:R11}; thence, we know
      \m{\im\apply{f}=\Dommenosa{0}=\set{a,\dots,n-1}}, so
      \m{f\preserves\Ruab}.
      Moreover, there is \m{x\in\Dommenosa{0,1}} due to \m{n\geq 3}.
      Thus, \m{x\preceq \top = 0}, but
      \m{x= f(x)\preceq f(0)=1 =\bot} would yield \m{x=\bot=1}, a
      contradiction.
\item Define \m{f\in\Op[1]{D}} as in~\eqref{item:Rn-2n-2}; thence, we
      recall \m{\im\apply{f}=\Dommenosa{n-1}=\set{0,\dotsc,b}}, so
      \m{f\preserves\Ruab}. We have
      \m{0\preceq\top=n-1}. As \m{n\geq 3}, we have \m{0<n-2}, and
      thus assuming \m{0=f(0)\preceq f(n-1)=n-2=\bot} would imply
      \m{0=\bot=n-2}, \ie\ \m{n=2}. Thus,
      \m{f\mathrel{\npres}\mathord{\preceq}}.
\end{enumerate}
\end{proof}

In Corollary~\ref{cor:a-b-geq2-reflexive} we have excluded inclusions
\m{\Pol{\Ruab}\subs\Pol{\mathord{\preceq}}} for bounded orders \m{\preceq},
whenever \m{a,b\in D} satisfy \m{a-b\geq 2}. In the previous lemma, a few
special cases have been considered. Now we deal with the rest using
transpositions fulfilling the criterion from
Lemma~\ref{lem:unary-f-pres-a-b-pres-Rab}.
\begin{proposition}\label{prop:bounded-orders}
Let \m{n\geq 3} and \m{\mathord{\preceq}\subs D^2} be a bounded order
relation on \m{D} with bottom element \m{\bot} and top \m{\top}. There do
not exist parameters \m{a,b\in D} such that
\m{\Pol{\Ruab}\subs\Pol{\mathord{\preceq}}}.
\end{proposition}
\begin{proof}
Corollary~\ref{cor:a-b-geq2-reflexive} excludes inclusions for
\m{a-b\geq 2}.
For the remainder of the proof let us suppose \m{a-b\leq 1}, \ie\
\m{a\leq b+1}. We shall exhibit unary operations (mostly transpositions)
that obviously do not preserve \m{\mathord{\preceq}}, but preserve
\m{\Ruab} (usually due to Lemma~\ref{lem:unary-f-pres-a-b-pres-Rab}).
For this we distinguish three cases regarding \m{\bot}.
First assume \m{\bot < a}. If there
exists \m{x<a} such that \m{x\neq \bot}, then we use the transposition
\m{\apply{x,\bot}}. Else all \m{x<a} satisfy \m{x=\bot}, \ie\
\m{\bot=0<a=1}. In this case we have \m{\top\neq \bot = 0},
so \m{\top\geq 1=a}. First consider the situation that \m{\top\leq b}.
If there exists some \m{x\in \interval{a}{b}\setminus\set{\top}}, we use the
transposition \m{\apply{x,\top}}. Otherwise,
\m{\interval{a}{b}\subs\set{\top}}, thus \m{1=a=\top=b} and \m{0=\bot}, which
is handled by Lemma~\ref{lem:orders-exceptions}\eqref{item:R11}. The
complementary case is that \m{\top>b}. If there exists \m{x>b} such that
\m{x\neq \top}, then we can use \m{\apply{x,\top}}, else every \m{x>b}
equals \m{\top}, and so we have \m{\top=n-1>b=n-2} together with
\m{a=1>0=\bot}. This is dealt with in
Lemma~\ref{lem:orders-exceptions}\eqref{item:12}.
\par
The second main case is when \m{a\leq \bot \leq b}. If there is some
\m{a\leq x\leq b} such that \m{x\neq \bot}, then we use
\m{\apply{x,\bot}}. Otherwise, \m{\interval{a}{b}\subs\set{\bot}}, and so
\m{a=\bot =b}. Due to \m{n\geq 3}, we have again \m{\top\neq\bot = a=b}.
Let us consider the situation \m{\top<a}. If there exists some \m{x<a},
\m{x\neq \top}, then we may use \m{\apply{x,\top}}, else every \m{x<a}
equals \m{\top}, so \m{\top=0<a=1=b=\bot}. This possibility is treated in
Lemma~\ref{lem:orders-exceptions}\eqref{item:21}. The opposite situation
is that \m{\top> a = b}. If there exists some \m{x>b}, \m{x\neq\top}, then
we use \m{\apply{x,\top}}, otherwise every \m{x>b} equals \m{\top}, and so
\m{\top=n-1>b=n-2=a=\bot}, which is solved in case~\eqref{item:22} of
Lemma~\ref{lem:orders-exceptions}.
\par
Third, let us deal with the possibility that \m{\bot>b}. If there is
some \m{x>b}, \m{x\neq\bot}, then we can use the transposition
\m{\apply{x,\bot}}. Otherwise, every \m{x>b} equals~\m{\bot}, so
\m{\bot=n-1>b=n-2}. Due to \m{n\geq 3}, we have \m{\top\neq \bot =n-1},
\ie\ \m{\top\leq n-2= b}. The first subcase is that \m{\top <a}. If there
exists some \m{x<a}, \m{x\neq \top}, we use the transposition
\m{\apply{x,\top}}. Else, all \m{x<a} satisfy \m{x=\top}, so
we obtain \m{\top=0<a=1}, \m{b=n-2<\bot =n-1}, which is treated in
Lemma~\ref{lem:orders-exceptions}\eqref{item:31}. The remaining subcase is
that \m{a\leq \top\leq b}. If there exists some \m{a\leq x\leq b},
\m{x\neq \top}, we use again \m{\apply{x,\top}}, else
\m{\interval{a}{b}\subs\set{\top}}, so \m{a=\top=b=n-2<n-1=\bot}, which has
been dealt with in Lemma~\ref{lem:orders-exceptions}\eqref{item:Rn-2n-2}.
\par
So in the case that \m{a-b\leq 1}, we have always found a transposition or a
unary operation as constructed in Lemma~\ref{lem:orders-exceptions} that
preserves \m{\Ruab}, but does not preserve the order \m{\preceq}.
Therefore, we have \m{\Pol{\Ruab}\not\subseteq\Pol{\mathord{\preceq}}}.
\end{proof}

\section{The case of \texorpdfstring{non\dash{}trivial}{non-trivial}
         equivalence relations}%
\label{sect:nontriv-equivalences}
Throughout this section, we shall employ the notation
\m{\Eq{D}} for the set of all equivalence relations on \m{D}.
It is our aim to show that maximal \cclones\xspace \m{\Pol{\Ruab}} are
contained in a maximal clone given by a non\dash{}trivial equivalence
relation if and only if \m{a=b+1}. In this case the equivalence relation
is uniquely determined.
\par

As our first result, we provide a simple sufficient condition for an
inclusion in a maximal clone described by an equivalence relation.
\begin{lemma}\label{lem:the-only-equivalence-rel}
Let \m{a,b\in D} satisfy \m{a=b+1} and \m{\theta\in \Eq{D}} be the
equivalence relation on \m{D} having the partition
\m{\factorBy{D}{\theta}=\set{\set{0,\dotsc,b},\set{a,\dotsc,n-1}}}. Then
we have
\m{\theta=\Ruab\cap\apply{\Ruab}^{-1}\in\genRelClone{\Ruab}}, and so the
inclusion \m{\Pol{\Ruab}\subs\Pol{\theta}} holds.
\end{lemma}
\begin{proof}
For any \m{\apply{x,y}\in D^2} we have \m{\apply{x,y}\in\theta} if and
only if \m{x,y\leq b} or \m{x,y\geq a}, \ie\ exactly if
\m{\apply{x,y}\in\set{0,\dotsc,b}^2\cup\set{a,\dotsc,n-1}^2
   =\Ruab\cap\apply{\Ruab}^{-1}}\ovflhbx{0.80713pt}
(cp.\ Lemma~\ref{lem:Rab-entails-0b-an-1}).
\end{proof}

In the remainder of this section we will prove that the situation
described in Lemma~\ref{lem:the-only-equivalence-rel} is the only one,
where a maximal \cclone\xspace can be contained in a maximal clone given
by a non\dash{}trivial equivalence relation.
\par
As a first step, we establish a few necessary conditions.
\begin{lemma}\label{lem:necessary-conds-for-equivalences}
Let \m{a,b\in D} and \m{\theta\in\Eq{D}\setminus\set{\Delta,\nabla}} be a
non\dash{}trivial equivalence relation such that
\m{\Pol{\Ruab}\subs\Pol{\theta}}. Then the following conditions are
fulfilled:
\begin{enumerate}[(a)]
\item\label{item:a-leq-b}
      \m{0<a\leq b+1\leq n-1}.
\item\label{item:xy-not-theta-implies-singletons}
      For every set \m{I\in
      \set{\set{0,\dotsc,a-1},\set{a,\dotsc,b},\set{b+1,\dotsc,n-1}}} we
      have
      \begin{equation*}
        \forall x,y\in I\colon \apply{x,y}\notin\theta \implies
        \abs{\fapply{x}_{\theta}} = 1 = \abs{\fapply{y}_{\theta}}.
      \end{equation*}
\item\label{item:xyz}
      For all \m{x,y,z\in D} where
      \m{\apply{x,y}\in\theta\setminus\Delta},
      we have the implication
      \begin{equation*}
      \apply{x,z\geq a \,\vee\, x,z\leq b\,\vee\, y,z\geq a
             \,\vee\, y,z\leq b}
      \implies \apply{x,z}\in\theta.
      \end{equation*}
\item\label{item:x-leq-b-y-geq-a-implies-x=y}
      \m{\forall x\leq b\,\forall y\geq a\colon \apply{x,y}\in\theta
         \implies b\geq x= y \geq a}.
\item\label{item:middle-singletons}
      \m{\forall a\leq x\leq b\colon \fapply{x}_{\theta}=\set{x}}.
\item\label{item:x-less-a-theta-class}
      \m{\forall x<a\colon \fapply{x}_{\theta}\subs\set{0,\dotsc,a-1}}.
\item\label{item:y-greater-b-theta-class}
      \m{\forall y>b\colon \fapply{y}_{\theta}\subs\set{b+1,\dotsc,n-1}}.
\item\label{item:impl-0-class-not-big}
      If\/ \m{\fapply{0}_{\theta}\neq\set{0,\dotsc,a-1}}, then we have
      \m{a-1>0}, \m{b+1<n-1}, \m{\fapply{x}_{\theta}=\set{x}} for all
      \m{x\leq b}, and
      \m{\fapply{n-1}_{\theta}=\set{b+1,\dotsc,n-1}}.
\item If\/ $\fapply{n-1}_{\theta}\neq\set{b+1,\dotsc,n-1}$, then we have
      $a-1>0$, $b+1<n-1$, \m{\fapply{y}_{\theta}=\set{y}} for all
      \m{y\geq a}, and
      \m{\fapply{0}_{\theta}=\set{0,\dotsc,a-1}}.
\end{enumerate}
\end{lemma}
\begin{proof}
\begin{enumerate}[(a)]
\item If \m{a=0}, or \m{b>n-2}, \ie\ \m{b=n-1}, then we would have a
      trivial clausal relation \m{\Ruab=D^2}, and so
      \m{\Pol{\Ruab}=\Op{D}} would make the inclusion
      \m{\Pol{\Ruab}\subs\Pol{\theta}} impossible. Moreover, if we had
      \m{a-b>1}, then Corollary~\ref{cor:a-b-geq2-reflexive} would imply
      the contradiction \m{\Pol{\Ruab}\not\subs\Pol{\theta}}. Therefore,
      it follows \m{0\neq a \leq b+1\leq n-1}.
\item Suppose, for a contradiction, that there exists a set
      \begin{equation*}
      I\in S\defeq
          \set{\set{0,\dotsc,a-1},\set{a,\dotsc,b},\set{b+1,\dotsc,n-1}}
      \end{equation*}
      and \m{x,y\in I} such that the stated implication fails. So we have
      \m{\apply{x,y}\notin\theta}, and since this assumption is symmetric,
      no generality is lost in assuming that
      \m{\abs{\fapply{x}_{\theta}}>1}. Let
      \m{z\in\fapply{x}_{\theta}\setminus\set{x}}, and define
      \m{f\in\Op[1]{D}} by \m{f\apply{x}\defeq y} and \m{f\apply{u}=u} for
      \m{u\neq x}. Obviously, \m{\apply{z,x}\in\theta}, but
      \m{\apply{f\apply{z},f\apply{x}}=\apply{z,y}\notin\theta}, as
      otherwise \m{\apply{x,z}\in\theta} and transitivity would imply
      \m{\apply{x,y}\in\theta}. Thus, \m{f\mathrel{\npres}\theta}.
      Moreover, as
      \m{x,y\in I}, we have \m{f\in\Pol{S}}, which implies that
      \m{f\preserves\Ruab} by Lemma~\ref{lem:unary-f-pres-a-b-pres-Rab}
      and statement~\eqref{item:a-leq-b}. This proves
      \m{f\in\Pol{\Ruab}\setminus\Pol{\theta}} in contradiction to
      \m{\Pol{\Ruab}\subs\Pol{\theta}}, so our initial assumption was
      false. Hence the claim holds.
\item Let \m{x,y,z\in D} where \m{\apply{x,y}\in\theta} and \m{x\neq y}.
      Moreover, the assumption of the implication is that we can find
      \m{w\in\set{x,y}} such that \m{w,z\geq a} or \m{w,z\leq b}. We
      define \m{f\in\Op[1]{D}} by \m{f\apply{w}\defeq w} and
      \m{f\apply{u}\defeq z} for \m{u\neq w}. Clearly, we have
      \m{\im\apply{f}= \set{w,z}}, so
      \m{\im\apply{f}\subs\set{a,\dotsc,n-1}} or
      \m{\im\apply{f}\subs\set{0,\dotsc,b}}. This implies
      \m{f\in\Pol{\Ruab}\subs\Pol{\theta}} by
      Lemma~\ref{lem:pres-of-max-clausal-rel} and the assumption of this
      lemma. So we get \m{\apply{f\apply{x},f\apply{y}}\in\theta} from
      \m{\apply{x,y}\in\theta}. If \m{w=x}, this means
      \m{\apply{x,z}\in\theta}. Else, if \m{w=y}, we obtain
      \m{\apply{z,y}\in\theta}, which together with
      \m{\apply{x,y}\in\theta} yields \m{\apply{x,z}\in\theta}.
\item Let us assume, for a contradiction, that there exists \m{x\leq b}
      and \m{y\geq a}, where the stated implication fails, \ie\ where
      \m{\apply{x,y}\in\theta}, but \m{x\neq y}. Now for every
      \m{z\geq a}, statement~\eqref{item:xyz} implies
      \m{\apply{x,z}\in\theta}, so
      \m{\set{a,\dotsc,n-1}\subs\fapply{x}_{\theta}}. Any other
      element \m{z\in D} satisfies \m{z<a\leq b+1} by
      item~\eqref{item:a-leq-b}, \ie\ \m{z\leq b}. Then again
      statement~\eqref{item:xyz} implies \m{\apply{x,z}\in\theta}. In
      conclusion, we have \m{D\subs \fapply{x}_{\theta}}, which means
      \m{\theta=\nabla}. As this was excluded beforehand, the claim
      holds.
\item Let us consider any \m{x\in D} where \m{a\leq x\leq b}. For
      \m{y\in\fapply{x}_{\theta}} such that \m{y\geq a}, we get \m{y=x} by
      item~\eqref{item:x-leq-b-y-geq-a-implies-x=y}. Any other
      \m{y\in\fapply{x}_{\theta}} satisfies \m{y<a\leq b+1}
      by~\eqref{item:a-leq-b}, \ie\ \m{y\leq b}. Again,
      statement~\eqref{item:x-leq-b-y-geq-a-implies-x=y}, with roles of
      \m{x} and \m{y} interchanged, yields \m{y=x}.
\item Let \m{x< a\leq b+1} (by~\eqref{item:a-leq-b}), then \m{x\leq b}. If
      there existed some \m{y\in\fapply{x}_{\theta}} such that
      \m{y\geq a}, then statement~\eqref{item:x-leq-b-y-geq-a-implies-x=y}
      would imply \m{a>x=y\geq a}. This contradiction proves
      \m{\fapply{x}_{\theta}\subs\set{0,\dotsc,a-1}}.
\item The proof is dual to that of
      statement~\eqref{item:x-less-a-theta-class}, using
      again~\eqref{item:a-leq-b}
      and~\eqref{item:x-leq-b-y-geq-a-implies-x=y}.
\item Assume \m{\fapply{0}_{\theta}\neq\set{0,\dotsc,a-1}}.
      Since~\eqref{item:a-leq-b} and~\eqref{item:x-less-a-theta-class}
      imply \m{\fapply{0}_{\theta}\subs\set{0,\dotsc,a-1}}, there must
      exist some \m{x<a} such that \m{x\notin\fapply{0}_{\theta}}. In
      particular, \m{x\neq 0} holds, so \m{0<x\leq a-1} yields \m{0<a-1}.
      Since \m{\apply{x,0}\notin\theta}, we get
      \m{\abs{\fapply{0}_{\theta}}=1}
      from~\eqref{item:xy-not-theta-implies-singletons}. So every
      \m{0<z<a} satisfies \m{\apply{0,z}\notin\theta},
      whence~\eqref{item:xy-not-theta-implies-singletons} yields
      \m{\abs{\fapply{z}_{\theta}}=1}. Together with
      statement~\eqref{item:middle-singletons} we can infer
      \m{\fapply{z}_{\theta} = \set{z}} for all \m{z\leq b}. Since
      \m{\theta\neq\Delta} by assumption, we cannot only have singleton
      equivalence classes for all other \m{y>b}. Thus, there must be some
      \m{y>b} where \m{\abs{\fapply{y}_{\theta}}>1}. If there were also
      some \m{z>b} such that \m{\apply{z,y}\notin\theta}, then
      again~\eqref{item:xy-not-theta-implies-singletons} would imply the
      contradiction \m{\abs{\fapply{y}_{\theta}}=1}. Hence, for all
      \m{z>b} we have \m{z\in\fapply{y}_{\theta}}, \ie\
      \m{\set{b+1,\dotsc,n-1}\subs\fapply{y}_{\theta}
                             \subs\set{b+1,\dotsc,n-1}}
      by~\eqref{item:y-greater-b-theta-class}. This means
      \m{\fapply{y}_{\theta}=\set{b+1,\dotsc,n-1}=\fapply{n-1}_{\theta}},
      and since \m{\abs{\fapply{y}_{\theta}}\geq 2}, we also get
      \m{b+1<n-1}.
\item The proof of this statement works dually to the preceding one.
\end{enumerate}
\end{proof}

We have gathered now enough prerequisites to prove the following result.
\begin{proposition}\label{prop:char-nontriv-eq}
Let \m{a,b\in D} and \m{\theta\in\Eq{D}\setminus\set{\Delta,\nabla}} be a
non\dash{}trivial equivalence relation. Then we have
\[
\Pol{\Ruab}\subs\Pol{\theta} \iff
a = b+1 \text{ and }
\factorBy{D}{\theta}=\set{\set{0,\dotsc,b},\set{a,\dotsc,n-1}}.
\]
\end{proposition}
\begin{proof}
The implication ``\m{\Longleftarrow}'' is stated in
Lemma~\ref{lem:the-only-equivalence-rel}. Conversely, let us assume that
\m{\Pol{\Ruab}\subs\Pol{\theta}}. For the remainder of the proof we can
suppose \m{0<a\leq b+1\leq n-1} due to
Lemma~\ref{lem:necessary-conds-for-equivalences}%
\eqref{item:a-leq-b}.
We define \m{f\in\Op[2]{D}}
by \m{f\apply{b+1,0}\defeq 0}, \m{f\apply{x,y}\defeq a} if \m{x,y>b} and
\m{f\apply{x,y}\defeq b} else. If \m{x\leq b} or \m{y\leq b}, then
\m{f\apply{x,y}\neq a}, so \m{f\apply{x,y}\leq b}.
Moreover, if \m{x,y\geq a}, then either \m{x,y>b} and \m{f\apply{x,y}=a},
or else \m{x,y\geq a>0} and \m{a\leq x\leq b} or \m{a\leq y\leq b}, whence
\m{f\apply{x,y}=b\geq a}. Therefore, the conditions of
Lemma~\ref{lem:binary-f-pres-Rab} are fulfilled, and so
\m{f\in\Pol{\Ruab}}.
\par

Now, we want to prove that
\m{\fapply{0}_{\theta}=\set{0,\dotsc,a-1}}. If this were false, then by
Lemma~\ref{lem:necessary-conds-for-equivalences}%
\eqref{item:impl-0-class-not-big} we would get \m{a-1>0}, \m{b+1<n-1},
\m{\fapply{x}_{\theta}=\set{x}} for every \m{x\leq b} and
\m{\fapply{n-1}_{\theta}=\set{b+1,\dotsc,n-1}}.
Thus, we have
\m{\apply{b+1,n-1},\apply{0,0}\in\theta}, but since \m{n-1\neq b+1}, we
obtain the tuple \m{\apply{f\apply{b+1,0},f\apply{n-1,0}}=\apply{0,b}},
which does not belong to \m{\theta} as
\m{b\notin\fapply{0}_{\theta}=\set{0}}.
Hence, \m{f\notin\Pol{\theta}}, in contradiction to the
assumed inclusion \m{\Pol{\Ruab}\subs\Pol{\theta}}.
\par
Consequently, we get \m{\fapply{0}_{\theta}=\set{0,\dotsc,a-1}}, and
dually, one can demonstrate
that \m{\fapply{n-1}_{\theta}=\set{b+1,\dotsc,n-1}}. If we can show
\m{a=b+1}, we will be done. As we already know \m{a\leq b+1},
we only have to exclude \m{a<b+1}, \ie\ \m{a\leq b}.
So, in order to obtain a contradiction, we suppose \m{b\geq a}. Then
we have \m{b\notin\fapply{0}_{\theta}=\set{0,\dotsc,a-1}}, \ie\
\m{\apply{0,b}\notin\theta}. If \m{b+1<n-1}, we could use the same
arguments as in the previous paragraph to prove that
\m{f\in\Pol{\Ruab}\setminus\Pol{\theta}}. Hence, we must have
\m{b+1=n-1}, and so $\fapply{y}_{\theta}=\set{y}$ holds for all
$y\geq a$ (recall
Lemma~\ref{lem:necessary-conds-for-equivalences}%
\eqref{item:middle-singletons}).
As $\fapply{0}_{\theta}=\penalty10000\set{0,\dotsc,a-1}$,
                            it follows \m{a-1>0} due
to \m{\theta\neq \Delta}. In this case we can use the dual version of
\m{f} to get a contradiction: define \m{g\in\Op[2]{D}} by
\m{g\apply{a-1,n-1}\defeq n-1}, \m{g\apply{x,y}\defeq b} if \m{x,y<a}, and
\m{g\apply{x,y}\defeq a} else. This function preserves \m{\Ruab} since the
conditions of Lemma~\ref{lem:binary-f-pres-Rab} are met: if \m{x\geq a} or
\m{y\geq a}, then \m{g\apply{x,y}\neq b}, so \m{g\apply{x,y}\geq a}. If
\m{x,y\leq b}, then \m{y< n-1}, so \m{g\apply{x,y}\neq n-1}. So either
\m{x,y<a}, whence \m{g\apply{x,y}=b}, or \m{a\leq x\leq b} or
\m{a\leq y\leq b} such that we get \m{g\apply{x,y}=a\leq b}. Thus,
\m{g\preserves\Ruab}. We finish by demonstrating that
\m{g\mathrel{\npres}\theta}.
Indeed, \m{\apply{0,a-1},\apply{n-1,n-1}\in\theta}, but due to
\m{a\leq b<n-1}, we have \m{a\notin\set{n-1}=\fapply{n-1}_{\theta}}.
So  we obtain that \m{g\mathrel{\npres}\theta} because
\m{\apply{g\apply{0,n-1},g\apply{a-1,n-1}}=\apply{a,n-1}\notin\theta}.
\par
This contradicts \m{\Pol{\Ruab}\subs\Pol{\theta}}, whence \m{a>b}, \ie\
\m{a = b+1}, follows.
\end{proof}

\section{The case of central relations}\label{sect:central-rels}
Inclusions \m{\Pol{\Ruab}\subs\Pol{\rho}} for at least ternary central
relations \m{\rho} have already been excluded
in Corollary~24 of~\cite{BehVarMaxCclonesPreprint}.
Moreover, unary central relations have been studied
in Section~\ref{sect:unary-rels}. So further in this section, we will only
consider binary central relations \m{\rho}. These are reflexive in the
usual sense, \ie\ \m{\Delta\subs \rho}, and hence, we can apply
Corollary~\ref{cor:a-b-geq2-reflexive}, which states
\m{\Pol{\Ruab}\not\subseteq \Pol{\rho}} for \m{a-b\geq 2} and
non\dash{}trivial \m{\rho}.
Next, we prove the same for \m{a-b=1}.

\begin{lemma}\label{lem:bin-central-rels-c<a-c>b}
Let \m{a\in\Dommenosa{0}}, \m{b\in\Dommenosa{n-1}} be such that
\m{a-b\leq 1}, and consider a non\dash{}trivial binary central relation
\m{\rho\subsetneq D^2} having a central element \m{c\in D} satisfying
\m{c<a} or \m{c>b}. Then there exists a function
\m{f\in\Pol[2]{\Ruab}\setminus\Pol{\rho}}.
\end{lemma}
\begin{proof}
If \m{c<a} then choose \m{d>b}, \eg\ \m{d=n-1}, else, if \m{c>b}, then
choose \m{d<a}, \eg\ \m{d=0}.
Moreover, let \m{\apply{u,v}\in D^2\setminus\rho}. We will consider three
cases, (1)~that \m{u,v\leq b}, (2)~\m{u,v\geq a}, which is not disjoint
from the previous case, and (3)~that neither~(1) nor~(2) holds. In
case~(3) no generality is lost in assuming \m{u<a\leq b+1}, \ie\
\m{u\leq b}, otherwise one can just swap \m{u} and \m{v} due to \m{\rho}
being symmetric. Since we are not in case~(1), we cannot have \m{v\leq b},
hence \m{v>b\geq a-1}, \ie\ \m{v\geq a}. So (3) means \m{u\leq b} and
\m{v\geq a}. In this case we define \m{z\defeq c}. For (1) we choose
\m{z\in\set{c,d}} such that \m{z<a}, implying \m{z\leq a-1\leq b}, and in
case~(2) we pick \m{z\in\set{c,d}} such that \m{z>b},
\ie\ \m{z\geq b+1\geq a}. We define now an operation \m{f\in\Op[2]{D}}. In
case~(1) we put \m{f\apply{x,y}\defeq \min\apply{x,y}} if \m{x,y\geq a},
\m{f\apply{x,y}\defeq v} if \m{\apply{x,y}=\apply{c,z}}, and
\m{f\apply{x,y}\defeq u} else. In case~(2) we set \m{f\apply{x,y}\defeq
\max\apply{x,y}} if \m{x,y\leq b}, \m{f\apply{x,y}\defeq v} if
\m{\apply{x,y}=\apply{c,z}}, and \m{f\apply{x,y}\defeq u} else. For~(3)
put \m{f\apply{x,y}\defeq\max\apply{x,y}} if \m{x,y\leq b} and
\m{\apply{x,y}\neq\apply{c,z}}, \m{f\apply{x,y}\defeq u} if
\m{\apply{x,y}=\apply{c,z}(=\apply{c,c})}, and \m{f\apply{x,y}\defeq v}
else, provided that \m{c<a}. Otherwise, if \m{c>b} in case~(3), we define
\m{f\apply{x,y}\defeq \min\apply{x,y}} if \m{x,y\geq a} and
\m{\apply{x,y}\neq \apply{c,z}}, \m{f\apply{x,y}\defeq v} if
\m{\apply{x,y}=\apply{c,z}(=\apply{c,c})}, and \m{f\apply{x,y}\defeq u}
else. It is not hard to check that always the function is
well\dash{}defined and that \m{f\in\Pol{\Ruab}} by
Lemma~\ref{lem:binary-f-pres-Rab}. Since \m{\rho} is reflexive and \m{c}
is a central element, we have \m{\apply{c,d},\apply{z,z}\in\rho}. However,
\m{\apply{f\apply{c,z},f\apply{d,z}}=\apply{u,v}\notin\rho} for case~(3)
and \m{c<a}, and otherwise we have
$(f(c,z),f(d,z))=(v,u)\notin\rho$ by symmetry
of \m{\rho}. This shows that $f\notin\Pol{\rho}$.\,{}
\end{proof}

\begin{corollary}\label{cor:bin-central-rels-a=b+1}
Let \m{a,b\in D} such that \m{a-b=1} and \m{\rho\subsetneq D^2} be any
non\dash{}trivial binary central relation, then there exists a function
\m{f\in \Pol[2]{\Ruab}\setminus\Pol{\rho}}.
\end{corollary}
\begin{proof}
Clearly \m{a-b=1} implies \m{a=b+1\geq 1>0} and \m{b=a-1<a\leq n-1}.
Moreover, \m{\rho} must have a central element \m{c\in D}.
We either have \m{c\geq a = b+1>b} or \m{c<a}. In both cases,
Lemma~\ref{lem:bin-central-rels-c<a-c>b} yields the result.
\end{proof}

The following lemma states conditions for an inclusion.
\begin{lemma}\label{lem:incl-binary-central}
Let \m{a,b\in D} such that \m{0<a\leq b<n-1}. Then we have
\begin{equation*}
\genRelClone{\Ruab}\ni \Ruab \cap \apply{\Ruab}^{-1}
= \bigcup_{a\leq c\leq b}
                  \apply{\set{0,\dotsc,c}^2\cup \set{c,\dotsc,n-1}^2}
\eqdef\sigma_{a,b},
\end{equation*}
and \m{\sigma_{a,b}\subs D^2\setminus\set{\apply{0,n-1},\apply{n-1,0}}}
is a non\dash{}trivial binary central relation having any
\m{c\in\set{a,\dotsc,b}} as a central element. Moreover, we have the
inclusion \m{\Pol{\Ruab}\subs\Pol{\sigma_{a,b}}}.
\end{lemma}
\begin{proof}
First, we demonstrate that \m{\sigma_{a,b}} is a non\dash{}trivial binary
central relation. It is clear that \m{\sigma_{a,b}} is symmetric as a
union of symmetric relations. Moreover, since \m{a\leq b}, there exists at
least one \m{c\in\set{a,\dotsc,b}}, \eg\ \m{c=a}. Now consider an arbitrary
such element \m{a\leq c\leq b}. If \m{x\in D} satisfies \m{x\leq c}, then
\m{\apply{x,x}\in\set{0,\dotsc,c}^2\subs\sigma_{a,b}}, else \m{x> c} and
\m{\apply{x,x}\in\set{c,\dotsc,n-1}^2\subs\sigma_{a,b}}. Hence,
\m{\Delta\subs\sigma_{a,b}}, \ie\ it is reflexive. Moreover, \m{c} is a
central element for \m{\sigma_{a,b}}, as for \m{x\leq c} the pairs
\m{\apply{x,c}} and \m{\apply{c,x}} belong to
\m{\set{0,\dotsc,c}^2\subs\sigma_{a,b}}, and
otherwise, we have \m{x>c} and
\m{\apply{x,c}} and \m{\apply{c,x}} lie in
\m{\set{c,\dotsc,n-1}^2\subs\sigma_{a,b}}.
Besides, we have \m{0<a\leq c}, so
\m{\apply{0,n-1}\notin\set{c,\dotsc,n-1}^2}, and neither have we
\m{\apply{0,n-1}\in\set{0,\dotsc,c}^2} due to \m{c\leq b<n-1}. As this is
true for any \m{a\leq c\leq b}, we obtain
\m{\apply{0,n-1}\notin\sigma_{a,b}}, which implies
\m{\apply{n-1,0}\notin\sigma_{a,b}} by symmetry of~\m{\sigma_{a,b}}.
\par

Next we prove that \m{\Ruab\cap\apply{\Ruab}^{-1}=\sigma_{a,b}}. Consider
any \m{a\leq c\leq b}. For all \m{x,y\leq c} we have \m{x,y\leq c\leq b}
and so \m{\apply{x,y},\apply{y,x}\in\Ruab}.
Dually, for all \m{x,y\geq c\geq a} we can infer
\m{\apply{x,y},\apply{y,x}\in\Ruab}, too. Therefore,
\m{\apply{x,y}\in\Ruab\cap\apply{\Ruab}^{-1}}.
\par

Conversely, suppose that \m{\apply{x,y}\in\Ruab\ni\apply{y,x}}. First,
consider the case that \m{x\leq b}. If \m{x\geq a}, too, then we have
\m{x,y\geq c} or \m{x,y\leq c} for \m{c\defeq x}. Else, we have \m{x<a},
which implies \m{y\leq b} due to \m{\apply{x,y}\in\Ruab}.
We consider two sub\dash{}cases: if \m{y\leq a}, then
\m{x,y\leq a\eqdef c}. Otherwise, we have \m{b\geq y>a} and put
\m{c\defeq y}. Then it follows \m{x<a\leq y=c} and \m{y\leq c}, finishing the
argument for the first case. Second, we have the possibility that
\m{x>b\geq a}. Then \m{\apply{y,x}\in\Ruab} implies \m{y\geq a}. Putting
\m{c\defeq a}, we have \m{x,y\geq a=c} in this case. Both times we have
shown that \m{\apply{x,y}\in\sigma_{a,b}}.
\par

The inclusion we have just demonstrated implies that
\m{\sigma_{a,b}\in\genRelClone{\Ruab}}, hence
\m{\Pol{\Ruab}=\Pol{\genRelClone{\Ruab}}\subs\Pol{\sigma_{a,b}}}.
\end{proof}

\begin{lemma}\label{lem:unary-f-mapping-complement-to-whole-complement}
Let \m{a,b\in D} such that \m{a\leq b} and \m{x_1,x_2 <a}, \m{y_1,y_2>b}.
Then we have \m{f\in\Pol[1]{\Ruab}} for \m{f\in\Op[1]{D}} defined by
\m{f\apply{x_1}\defeq x_2}, \m{f\apply{y_1}\defeq y_2}
and \m{f\apply{z}\defeq z} for \m{z\in\Dommenosa{x_1,y_1}}.
\end{lemma}
\begin{proof}
First, the function \m{f\in\Op[1]{D}} is
well\dash{}defined due to \m{x_1<a\leq b<y_1}. Since \m{x_1,x_2< a\leq b}
and \m{y_1>b}, it is evident that \m{f\preserves\set{0,\dotsc,b}}.
Similarly, we obtain that \m{f\preserves \set{a,\dotsc,n-1}}. Using
Lemma~\ref{lem:unary-f-pres-a-b-pres-Rab}, we can infer that
\m{f\in\Pol{\Ruab}}.
\end{proof}

With these lemmas at hand, we can prove the following characterisation.
\begin{proposition}\label{prop:incl-bin-central-iff-sigmaab}
Let \m{a,b\in D}, \m{\sigma_{a,b}\subs D^2} be defined as in
Lemma~\ref{lem:incl-binary-central} and \m{\rho\subsetneq D^2} be a
non\dash{}trivial binary central relation. Then we have
\begin{equation*}
\Pol{\Ruab}\subs \Pol{\rho} \iff 0<a\leq b<n-1 \text{ and }
\rho=\sigma_{a,b}.
\end{equation*}%
\end{proposition}%
\begin{proof}
The implication ``\m{\Longleftarrow}'' holds by
Lemma~\ref{lem:incl-binary-central}. Conversely, suppose that
\m{\Pol{\Ruab}\subs \Pol{\rho}} is true. Then \m{a\neq 0} and
\m{b\neq n-1}, as otherwise \m{\Ruab = D^2} and then
\m{\Pol{\Ruab}=\Op{D}}, which is not contained in any maximal clone.
Moreover, as \m{\rho} is reflexive and non\dash{}trivial,
Corollaries~\ref{cor:a-b-geq2-reflexive}
and~\ref{cor:bin-central-rels-a=b+1} allow us to infer that
\m{a\leq b}. It remains to show that \m{\rho=\sigma_{a,b}}.
\par

First, let us consider the inclusion \m{\sigma_{a,b}\subs\rho}. For this
let \m{d\in D} be a central element of \m{\rho}. If \m{d<a} or \m{d>b},
then this would violate the assumed inclusion
\m{\Pol{\Ruab}\subs\Pol{\rho}} due to
Lemma~\ref{lem:bin-central-rels-c<a-c>b}. Hence, we have \m{a\leq d\leq b}.
For any pair \m{\apply{x,y}\in\set{0,\dotsc,b}^2\cup\set{a,\dotsc,n-1}^2}
we can define a unary function \m{f\in\Op[1]{D}} by \m{f\apply{0}\defeq x}
and \m{f\apply{z}\defeq y} if \m{z\in\Dommenosa{0}}. Obviously, we have
\m{\im\apply{f} = \set{x,y}}, such that
\m{\im\apply{f}\subs\set{0,\dotsc,b}} or
\m{\im\apply{f}\subs\set{a,\dotsc,n-1}}.
So using Lemma~\ref{lem:pres-of-max-clausal-rel} we obtain
\m{f\in\Pol{\Ruab}\subs\Pol{\rho}}, and thus
\m{\apply{x,y}=\apply{f\apply{0},f\apply{d}}\in \rho} since \m{d\geq a>0}
was a central element of \m{\rho}. This demonstrates that
\m{\rho\sups\set{0,\dotsc,b}^2\cup\set{a,\dotsc,n-1}^2}. Evidently, the
latter set equals \m{\sigma_{a,b}}.
\par

To prove that \m{\rho\subs\sigma_{a,b}} we rule out that
\m{\apply{D^2\setminus\Ruab}\cap \rho\neq\emptyset}. Namely, if there
were some \m{\apply{x_1,y_1}\in\apply{D^2\setminus\Ruab}\cap\rho}, then
for every pair \m{\apply{x_2,y_2}\in D^2\setminus\Ruab}, we could use the
function \m{f\in\Pol{\Ruab}\subs \Pol{\rho}} constructed in
Lemma~\ref{lem:unary-f-mapping-complement-to-whole-complement} to show
that \m{\apply{x_2,y_2}=\apply{f\apply{x_1},f\apply{y_1}}\in\rho}. This
would mean \m{D^2\setminus\Ruab\subs\rho}, and, by symmetry of~\m{\rho},
would imply \m{D^2\setminus\apply{\Ruab}^{-1}\subs\rho}.
Hence, we would have the inclusion
\m{D^2\setminus\sigma_{a,b} =
D^2\setminus\apply{\Ruab\cap\apply{\Ruab}^{-1}} \subs\rho}. Together with
\m{\sigma_{a,b}\subs \rho}, we would get \m{\rho=D^2}, in contradiction to
\m{\rho} being non\dash{}trivial.
\par

Therefore, it holds \m{\apply{D^2\setminus\Ruab}\cap \rho=\emptyset}, which
means \m{\rho\subs\Ruab}. By symmetry of~\m{\rho} this implies
\m{\rho=\rho^{-1}\subs\apply{\Ruab}^{-1}}, and thus
\m{\rho\subs\Ruab\cap\apply{\Ruab}^{-1}=\sigma_{a,b}}.
\end{proof}

\section{Theorem statement}\label{sect:thm-statement}
We can combine the previously proven results to obtain the following
theorem, giving a complete description of the relationship between maximal
clones and maximal clausal clones.
\par

\begin{theorem}\label{thm:char-relationship-max-clones-max-cclones}
For every maximal \cclone\xspace \m{\Pol{\Ruab}} on
\m{D=\set{0,\dotsc,n-1}}, where \m{n\in\N}, and
\m{a\in\Dommenosa{0}} and \m{b\in\Dommenosa{n-1}}, there
exists precisely one maximal clone \m{M} such that \m{\Pol{\Ruab}\subs M}.
\par

More precisely, we have that
\begin{itemize}
\item
\m{\Pol{\Relab{(1)}{(0)}} = \Pol{\leq_{2}}} for \m{n=2};
\item
for \m{n\geq 3} the following inclusions hold:
\begin{itemize}
\item
\m{\Pol{\Ruab}\subs\Pol{\rho}} if \m{a-b>1},
where \m{\rho=\set{0,\dotsc,b}\cup\set{a,\dotsc,n-1}} is a unary
non\dash{}trivial relation;
\item
\m{\Pol{\Ruab}\subs\Pol{\theta}} if \m{a-b=1}, where \m{\theta} is the
equivalence relation on \m{D} given by the partition
\m{\factorBy{D}{\theta} = \set{\set{0,\dotsc,b},\set{a,\dotsc,n-1}}}; and
\item
\m{\Pol{\Ruab}\subs\Pol{\sigma_{a,b}}} if \m{a-b<1} where
\m{\sigma_{a,b}} denotes the binary central relation
\m{\set{0,\dotsc,b}^2\cup\set{a,\dotsc,n-1}^2}.
\end{itemize}
\end{itemize}
\end{theorem}
\begin{proof}
Summarising previous work,
inclusions \m{\Pol{\Ruab}\subs\Pol{\rho}} are
impossible whenever \m{\rho} is the graph of a prime permutation
(\cite[Lemma~20]{BehVarMaxCclonesPreprint}), an affine
relation corresponding to some elementary \name{Abel}ian \nbdd{p}group
(\cite[Lemma~21]{BehVarMaxCclonesPreprint}), an at least ternary
(non\dash{}trivial) central or \nbdd{h}regular relation
(\cite[Corollary~24]{BehVarMaxCclonesPreprint}),
or a bounded partial order relation for \m{n\geq 3} (Proposition~\ref{prop:bounded-orders}).
So from the types of relations listed in
Theorem~\ref{thm:Rosenberg-classification} only non\dash{}trivial
equivalence relations, bounded partial order
relations for \m{n=2} and unary and binary central relations remain.
\par

Lemma~\ref{lem:incl-unary-rels}
and Propositions~\ref{prop:char-nontriv-eq}
and~\ref{prop:incl-bin-central-iff-sigmaab} confirm the inclusions claimed
in the theorem for \m{n\geq 3}. We only have to prove that each maximal
\cclone\xspace is not contained in any other maximal clone. For instance,
if \m{a-b=1}, then Proposition~\ref{prop:incl-bin-central-iff-sigmaab} and
Lemma~\ref{lem:incl-unary-rels} show that \m{\Pol{\Ruab}} is not contained
in \m{\Pol{\rho}} for any non\dash{}trivial unary or binary central
relation \m{\rho}. Moreover, by Proposition~\ref{prop:char-nontriv-eq},
an inclusion \m{\Pol{\Ruab}\subs\Pol{\theta}}, where \m{\theta} is a
non\dash{}trivial equivalence relation, implies that \m{\theta} is exactly
the equivalence stated in the theorem. For the cases \m{a-b \gtrless 1}
analogous arguments prove that \m{\Pol{\Ruab}} is a subset of a unique
maximal clone.
\par

The statements concerning \m{\abs{D}=n=2} have been established already
in~\cite[Theorem~2.14]{Var} (see
also~\cite[Theorem~6]{BehVarMaxCclonesPreprint}): the clone of monotone
\name{Boole}an functions is the only maximal \cclone\xspace on a
two\dash{}element domain.
\end{proof}

From the previous theorem, we can derive a completeness criterion for
clones on finite sets described by clausal relations. This will require
the following additional lemma.
\begin{lemma}\label{lem:all-non-full-cclones-below-maximal-one}
Let \m{n\in\N}, \m{D=\set{0,\dotsc,n-1}} and \m{Q\subs \CR} be a set of
clausal relations. If\/ \m{\Pol{Q}\subsetneq \Op{D}}, then there is a
maximal \cclone\xspace\/ \m{\Pol{\Ruab}}
(\m{a\in\Dommenosa{0}},
\m{b\in\Dommenosa{n-1}}) such that \m{\Pol{Q}\subs\Pol{\Ruab}}.
\end{lemma}
\begin{proof}
If every \m{\Rab\in Q} contains a \m{0} among
\m{\set{a_1,\dotsc,a_p}} or \m{n-1\in\set{b_1,\dotsc,b_q}}, then
\m{\Pol{Q}=\Op{D}}, so the premise of the implication is not fulfilled.
This is in particular the case for \m{n\leq 1}, so let us further consider
\m{n\geq 2} and suppose that there exists some \m{\Rab\in Q} where
\m{\mathbf{a}\in\apply{\Dommenosa{0}}^p} and
\m{\mathbf{b}\in\apply{\Dommenosa{n-1}}^q}.
It follows that \m{\Pol{Q}\subs\Pol{\set{\Rab}}}.
By Lemma~6.1.3 of~\cite{Edith-thesis} we have
\m{\Pol{\set{\Rab}}\subs\Pol{\ovflhbx{0.42195pt}\set{\Ruab}}} where
\m{a=\min\set{a_1,\dotsc,a_p}>0} and
\m{b=\max\set{b_1,\dotsc,b_q}<n-1}.
By~Theorem~\ref{teo_caracterization_cclonemax}, \m{\Pol{\Ruab}} is
indeed a maximal \cclone; by the above, it is a superclone
of \m{\Pol{Q}}.
\end{proof}

\begin{corollary}\label{cor:completeness-criterion}
Let \m{Q\subs \CR} be a set of clausal relations on
\m{D=\set{0,\dotsc,n-1}}, \m{n\geq 3}, and put \m{F\defeq \Pol{Q}}. If for
each \m{0\leq b<n-1}  there is some \m{f\in F} such that
\m{f\mathrel{\npres}\theta_{b}}, where \m{\theta_{b}} is the equivalence
relation belonging to the non\dash{}trivial partition
\m{\factorBy{D}{\theta_{b}}=\set{\set{0,\dotsc,b},\set{b+1,\dotsc,n-1}}},
and for each \m{0<a\leq b<n-1} there is some \m{f\in F} such that
\m{f\npres\set{0,\dotsc,b}^{2}\cup\ovflhbx{2pt}\set{a,\dotsc,n-1}^{2}\ovflhbx{4pt}}, and for each
\m{0\ovflhbx{1pt}\leq\ovflhbx{1pt} b
    \ovflhbx{1.00pt}\leq\ovflhbx{1.00pt}
    n\ovflhbx{1.00pt}-\ovflhbx{1.00pt}3} and all \m{2\leq k\leq n-1-b}
we have
\m{f\mathrel{\npres}\set{0,\dotsc,b}\cup\set{b+k,\dotsc,n-1}} for some
\m{f\in F}; then \m{F=\Pol{Q}=\Op{D}}.
\end{corollary}
\begin{proof}
By the assumptions and
Theorem~\ref{thm:char-relationship-max-clones-max-cclones}, we have
\m{F\not\subs\Pol{\Ruab}} for all parameters \m{a\in\Dommenosa{0}},
\m{b\in\Dommenosa{n-1}}. Therefore, the \cclone\xspace \m{F} is not
contained in any maximal \cclone.
Using Lemma~\ref{lem:all-non-full-cclones-below-maximal-one},
we can conclude that \m{\Pol{Q}=F} must be the full \cclone\xspace
\m{\Op{D}}.
\end{proof}


\bibliographystyle{amsalpha}
\bibliography{maxCclones3}

\smallskip

\myContact{\CorrespondingAuthor}{%
\TUWname,
\InstitutCL,
\PostleitzahlWien}{behrisch@logic.at}

\myContact{\SecondAuthor}{%
\ULeeds,
\SchoolName,
\EdithAddress}{pmtemv@leeds.ac.uk}
\end{document}